\documentclass{amsart}
\usepackage{mathrsfs}
\usepackage{amssymb,latexsym,amsmath,amsthm,enumitem,mathrsfs}
\usepackage{hyperref}
\usepackage{color}
\usepackage{stmaryrd}
\usepackage{mathrsfs}
\usepackage{lineno}
\usepackage{bbm}   
\usepackage{scalerel,stackengine}
\usepackage{tikz}

\usepackage{ esint }
\usepackage{graphicx}
\usepackage{amssymb}
\newtheorem{theorem}{Theorem}
\newtheorem{lemma}{Lemma}
\newtheorem{remark}{Remark}

\newcommand{\cC}{\mathcal{C}}
\newcommand{\cD}{\mathcal{D}}

\newcommand{\cH}{\mathcal{H}}

\newcommand{\R}{\mathbb R}

\newcommand{\T}{\mathbb T}

\renewcommand{\S}{\mathbb S}

\newcommand{\e}{\varepsilon}

\newcommand{\s}{\psi}

\renewcommand{\d}{\delta}

\def\set4{\mathcal I}
\def\tup14{(1,2,3,4)}

\newcommand\vwidehat[1]{\arraycolsep=0pt\relax%
\begin{array}{c}
\stretchto{
  \scaleto{
    \scalerel*[\widthof{\ensuremath{#1}}]{\kern-.5pt\bigwedge\kern-.5pt}
    {\rule[-\textheight/2]{1ex}{\textheight}} 
  }{\textheight} %
}{0.5ex}\\           
#1\\                 
\rule{-1ex}{0ex}
\end{array}
}
\newtheorem*{comm*}{Comment}
\newtheorem{definition}{Definition}
\newtheorem*{lemma*}{Lemma}
\newtheorem{corollary}{Corollary}

\usepackage{bbm}   
\usepackage{scalerel,stackengine}
\stackMath
\newcommand\widecheck[1]{%
\savestack{\tmpbox}{\stretchto{%
  \scaleto{%
    \scalerel*[\widthof{\ensuremath{#1}}]{\kern-.6pt\bigwedge\kern-.6pt}%
    {\rule[-\textheight/2]{1ex}{\textheight}}
  }{\textheight}%
}{0.5ex}}%
\stackon[1pt]{#1}{\scalebox{-1}{\tmpbox}}%
}
\usepackage{ esint }
\newcommand{\supp}{\mathrm{supp}}

\newcommand{\de}{\delta}
\newcommand{\ga}{\gamma}
\newcommand{\Ga}{\Gamma}
\newcommand{\ZS}{\mathbb S}

\newcommand{\wh}{\widehat}

\newcommand{\si}{\sigma}
\newcommand{\Si}{\Sigma}
\newcommand{\Tau}{\mathcal T}

\newcommand{\en}{\epsilon_{\circ}}
\newcommand{\I}{\mathbb I}

\newcommand{\Gn}{\Ga_{\circ}}
\newcommand{\Id}{\boldsymbol 1}

\begin{document}

\keywords{decoupling inequalities, superlevel set}
\subjclass[2020]{42B15, 42B20}

\date{}

\title[restricted projections to lines in $\R^3$]{An exceptional set estimate for  restricted projections to lines in $\R^3$}
\author{Shengwen Gan}
\address{Department of Mathematics\\
Massachusetts Institute of Technology\\
Cambridge, MA 02142-4307, USA}
\email{shengwen@mit.edu}

\author{Larry Guth}
\address{Department of Mathematics\\
Massachusetts Institute of Technology\\
Cambridge, MA 02142-4307, USA}
\email{lguth@math.mit.edu}

\author{Dominique Maldague}
\address{Department of Mathematics\\
Massachusetts Institute of Technology\\
Cambridge, MA 02142-4307, USA}
\email{dmal@mit.edu}
\maketitle

\begin{abstract}
    Let $\ga:[0,1]\rightarrow \ZS^{2}$ be a non-degenerate curve in $\R^3$, that is to say, $\det\big(\ga(\theta),\ga'(\theta),\ga''(\theta)\big)\neq 0$. For each $\theta\in[0,1]$, let $l_\theta=\{t\ga(\theta):t\in\R\}$ and $\rho_\theta:\R^3\rightarrow l_\theta$ be the orthogonal projections. 
    We prove an exceptional set estimate. For any Borel set $A\subset\R^3$ and $0\le s\le 1$, define $E_s(A):=\{\theta\in[0,1]: \dim(\rho_\theta(A))<s\}$. We have $\dim(E_s(A))\le 1+\frac{s-\dim(A)}{2}$.  
\end{abstract}

\section{Introduction}

If $\ga:[0,1]\rightarrow\ZS^2$ is a curve that satisfies the non-degenerate condition $$\det\big(\ga(\theta),\ga'(\theta),\ga''(\theta)\big)\neq 0,$$ 
then 
we call $\ga$ a \textit{non-degenerate curve}.
A model example for the non-degenerate curve is
$\ga_\circ:\theta\mapsto (\frac{\cos\theta}{\sqrt{2}},\frac{\sin\theta}{\sqrt{2}},\frac{1}{\sqrt{2}})$ $(\theta\in[0,1])$. 

In this paper, we study the the projections in $\R^3$ whose directions are determined by $\ga$.
For each $\theta\in[0,1]$, let $V_\theta\subset \R^3$ be the $2$-dimensional subspace that is orthogonal to $\ga(\theta)$ and let $l_\theta\subset\R^3$ be the $1$-dimensional subspace spanned by $\ga(\theta)$.
We also define $\pi_\theta:\R^3\rightarrow V_\theta$ to be the orthogonal projection onto $V_\theta$, and define $\rho_\theta:\R^3\rightarrow l_\theta$ to be the orthogonal projection onto $l_\theta$. 
We use $\dim X$ to denote the Hausdorff dimension of set $X$. Let us state our main results.

\begin{theorem}\label{mainlineex}
Suppose $A\subset \R^3$ is a Borel set of Hausdorff dimension $\alpha$. For $0\le s\le 1$, define the exceptional set 
\begin{equation}\label{Es}
     E_s=\{\theta\in[0,1]: \dim(\rho_\theta(A))<s \}. 
\end{equation}
Then we have
\begin{equation}\label{exestimate}
    \dim (E_s)\le \max\{0,1+\frac{s-\alpha}{2}\}. 
\end{equation}
 
\end{theorem}
 
As a corollary, we have

\begin{corollary}\label{mainline}
Suppose $A\subset \R^3$ is a Borel set of Hausdorff dimension $\alpha$. Then we have
$$ \dim(\rho_\theta(A))=\min\{1,\alpha\},\textup{~for~a.e.~}\theta\in[0,1]. $$
\end{corollary}

\begin{remark}
{\rm
The proof of Theorem \ref{mainlineex} relies on the small cap decoupling for the general cone.
We also remark that, for the set of directions determined by the model curve $\ga_\circ$, K\"aenm\"aki, Orponen and Venieri can prove the exceptional set estimate with upper bound $\dim(E_s)\le\frac{\alpha+s}{2\alpha}$ when $\alpha\le 1$ (see \cite{kaenmaki2017marstrand} Theorem 1.3).
The novelty of our paper is that we prove a Falconer-type exceptional set estimate for general non-degenerate curve, hence Corollary \ref{mainline}.
}
\end{remark}

\begin{remark}
{\rm
Pramanik, Yang and Zahl \cite{zahl2022} have also recently proved Corollary \ref{mainline} with an exceptional set estimate of the form 
$\dim(E_s)\le s$, compared to \eqref{exestimate}.
Their proof is based on some incidence estimates for curves in the spirit of Wolff's circular maximal function estimate. The estimates in \cite{zahl2022} hold for curves that are only $C^2$, which requires a very different proof from earlier work of Wolff and others on these problems.
}
\end{remark}

\begin{remark}
{\rm
It is also an interesting question to ask for the estimate of the set
$$ E=\{\theta\in[0,1]: \cH^1(\rho_\theta(A))=0\} $$
which consists of directions to which the projection of $A$ has zero measure.
We notice that recently Harris \cite{harris2022length} proved that
\begin{equation}\label{harris}
    \dim (\{\theta\in[0,1]: \cH^1(\rho_\theta(A))=0\})\le \frac{4-\dim A}{3}. 
\end{equation} 
Intuitively, one may think of $E$ as $E_1$ ($E_1$ is defined in \eqref{Es}). The main result of this paper \eqref{exestimate} yields $\dim(E_1)\le \frac{3-\dim A}{2}$ which is better than the bound $\frac{4-\dim A}{3}$. This shows that \eqref{harris} cannot imply \eqref{exestimate}.

}
\end{remark}

\medskip

Now we briefly discuss the history of projection theory.
Projection theory dates back to Marstrand\cite{marstrand1954some}, who showed that if $A$ is a Borel set in $\R^2$, then the projection of $A$ onto almost every line through the origin has Hausdorff dimension $\min\{1,\dim A\}$. This was generalized to higher dimensions by Mattila\cite{mattila1975hausdorff}, who showed that if $A$ is a Borel set in $\R^n$, then the projection of $A$ onto almost every $k$-plane through the origin has Hausdorff dimension $\min\{k,\dim A\}$.
More recently, F\"assler and Orponen \cite{fassler2014restricted} started to consider the projection problems when the direction set is restricted to some submanifold of Grassmannian. Such problems are known as the restricted projection problem. F\"assler and Orponen made conjectures about restricted projections to lines and planes in $\R^3$ (see Conjecture 1.6 in \cite{fassler2014restricted}). In this paper, we give an answer to the conjecture about the projections to lines.

\section{Projection to one dimensional family of lines}\label{sectionline}

In this section, we prove Theorem \ref{mainline}. Theorem \ref{mainline} will be a result of an incidence estimate that we are going to state later. Recall that $\ga:[0,1]\rightarrow \S^2$ a non-degenerate curve.

\begin{definition}
For a number $\de>0$ and any set $X$, we use $|X|_\de$ to denote the maximal number of $\de$-separated points in $X$.
\end{definition}

\begin{definition}[$(\de,s)$-set]\label{deltaset}
Let $P\subset \R^n$ be a bounded set. Let $\de>0$ be a dyadic number, and let $0\le s\le d$. We say that $P$ is a $(\de,s)$-set if
$$ |P\cap B_r|_\de\lesssim  (r/\de)^s ,  $$
for any $B_r$ being a ball of radius $r$ with $\de\le r\le 1$.
\end{definition}

Let $\cH^t_\infty$ denote the $t$-dimensional Hausdorff content which is defined as
$$ \cH^t_\infty(B):=\inf\{ \sum_i r(B_i)^t: B\subset \cup_i B_i \}. $$
We recall the following  result (see \cite{fassler2014restricted} Lemma 3.13). 
\begin{lemma}\label{lemdelta}
Let $\de,s>0$, and $B\subset \R^n$ with $\cH_\infty^s(B):=\kappa>0$. Then there exists a $(\de,s)$-set $P\subset B$ with cardinality $\#P\gtrsim \kappa \de^{-s}$.
\end{lemma}

Next, we state a useful lemma. It is proved in \cite{gan2022restricted}, but we still provide the full details here. The lemma roughly says that given a set $X$ of Hausdorff dimension less than $s$, then we can find a covering of $X$ by squares of dyadic lengths which satisfy a certain $s$-dimensional condition. Let us use $\cD_{2^{-k}}$ to denote the lattice squares of length $2^{-k}$ in $[0,1]^2$.

\begin{lemma}\label{usefullemma}
Suppose $X\subset [0,1]^2$ with $\dim X< s$. Then for any $\e>0$, there exist dyadic squares $\cC_{2^{-k}}\subset \cD_{2^{-k}}$ $(k>0)$ so that 
\begin{enumerate}
    \item $X\subset \bigcup_{k>0} \bigcup_{D\in\cC_{2^{-k}}}D, $
    \item $\sum_{k>0}\sum_{D\in\cC_{2^{-k}}}r(D)^s\le \e$,
    \item $\cC_{2^{-k}}$ satisfies the $s$-dimensional condition: For $l<k$ and any $D\in \cD_{2^{-l}}$, we have $\#\{D'\in\cC_{2^{-k}}: D'\subset D\}\le 2^{(k-l)s}$.
\end{enumerate}
\end{lemma}

\begin{proof}[Proof of Lemma \ref{usefullemma}]
Consider all the covering $\cC$ of $X$ by dyadic lattice squares that satisfy condition (1), (2) in Lemma \ref{usefullemma}, i.e., $\cC\subset \bigcup_{k>0}\cD_{2^{-k}}$, $X\subset\bigcup_{D\in\cC}D$ and $\sum_{D\in\cC}r(D)^s\le \e$. We also assume all the dyadic squares in $\cC$ are disjoint. 
We will define an order ``$<$" between any two of such coverings $\cC, \cC'$. First, we define the $k$-th covering number of $\cC$ by
$$ c_k(\cC):=\#(\cC\cap \cD_{2^{-k}}), $$
which is the number of $2^{-k}$-squares in the covering $\cC$.

We say $\cC<\cC'$, if they satisfy: (1) There is a $k_0\ge 0$ such that
$\cC\cap\cD_{2^{-k}}=\cC'\cap\cD_{2^{-k}}$ ($k<k_0$), and $\cC\cap \cD_{2^{-k_0}}\subset \cC'\cap \cD_{2^{-k_0}}$; (2) For any $x\in X$, the square in $\cC'$ that covers $x$ contains the square in $\cC$ that covers $x$. 
It is not hard to check the transitivity: If $\cC<\cC'$ and $\cC'<\cC''$, then $\cC<\cC''$.

Suppose $\cC$ is a covering that is maximal with respect to the order $<$. Then we can show that $\cC$ satisfies condition $(3)$ in Lemma \ref{usefullemma}.
Suppose by contradiction, there exist $l<k$ and $D\in\cD_{2^{-l}}$ so that
\begin{equation}\label{bythis}
    \#\{D'\in \cC\cap\cD_{2^{-k}}:D'\subset D \}>2^{(k-l)s}. 
\end{equation} 
We define another covering $\cC'$ by adding $D$ to $\cC$ and deleting $\{D'\in\cC\cap\cD_{2^{-k}}:D'\subset D\}$ from $\cC$.
It is easy to check that $\cC'$ is still a covering of $X$. By \eqref{bythis}, we can also check $\sum_{D\in\cC'}r(D)^s<\sum_{D\in\cC}r(D)^s\le \e$, so $\cC'$ satisfies $(2)$ in Lemma \ref{usefullemma}. However, $\cC<\cC'$ which contradicts the maximality of $\cC$. 

Now, it suffices to find a maximal element among all the coverings that satisfy condition $(1), (2)$ in Lemma \ref{usefullemma}. First of all, such covering exists by the definition of Hausdorff dimension and $\dim X<s$. By Zorn's lemma, it suffices to find an upper bound for any ascending chain.

Let $\{ \mathcal{C}_j\}_{j \in J}$ be an infinite chain of coverings of $X$. Define 
\[ \mathcal{C} = \bigcap_{j \in J} \bigcup_{\substack{ i \in J \\ \mathcal{C}_i \geq \mathcal{C}_j } } \mathcal{C}_i. \]
We show that $\cC$ is a maximal element of the chain. First, we show that $\cC$ covers $X$.
For $x \in X$, let $D^{(j)}$ be the largest dyadic square in $\bigcup_{\substack{ i \in J \\ \mathcal{C}_i \geq \mathcal{C}_j } } \mathcal{C}_i$ containing $x$. By the definition of the partial order and the fact that chains are totally ordered, $D^{(j)}$ is independent of $j$, and thus $D^{(j)} \in \mathcal{C}$. This shows that $\mathcal{C}$ is a covering of $X$. Let $K \in \mathbb{N}$. Choose $j \in J$ such that $\mathcal{C}_i \cap \mathcal{D}_{2^{-k}} = \mathcal{C}_j \cap \mathcal{D}_{2^{-k}}$ for all $0 \leq k \leq K$ and all $\mathcal{C}_i \geq \mathcal{C}_j$. Then 
\[ \sum_{k=0}^K \sum_{D \in \mathcal{C} \cap \mathcal{D}_{2^{-k} }} r(D)^s \leq \sum_{k=0}^K \sum_{D \in \mathcal{C}_j \cap \mathcal{D}_{2^{-k} }} r(D)^s \leq \varepsilon. \]
Letting $K \to \infty$ gives 
\[ \sum_{D \in \mathcal{C}} r(D)^s \leq \varepsilon. \]
So, $\cC$ satisfies condition $(2)$. By definition, it is easy to check $\cC_i\le \cC$ for every $\cC_i$ in the initial chain. This proves that $\cC$ is a maximal element. \end{proof}

\begin{remark}
{\rm
Besides $[0,1]^2$, this lemma holds for other compact metric spaces, for example $[0,1]^n$ or  $\S^2$. The proof is exactly the same.
}
\end{remark}

Our main effort will be devoted to the proof of the following theorem.

\begin{theorem}\label{linediscrt} Fix $0<s<1$. For each $\e>0$, there exists $C_{s,\e}$ so that the following holds. Let $\d>0$. Let $H\subset B^3(0,1)$ be a union of disjoint $\de$-balls and we use $\#H$ to denote the number of $\de$-balls in $H$. Let $\Theta$ be a $\de$-separated subset of $[0,1]$ such that $\Theta$ is a $(\de,t)$-set and
$\#\Theta\gtrsim (\log\de^{-1})^{-2}\de^{-t}$ for some $t>0$. Assume for each $\theta\in \Theta$, we have a collection of $\de\times 1\times 1$-slabs $\S_\theta$ with normal direction $\ga(\theta)$. $\S_\theta$ satisfies the $s$-dimensional condition: 
\begin{enumerate}
    \item $\#\S_\theta\lesssim \de^{-s}$, 
    \item $\#\{S\in\S_\theta:S\cap B_r\}\lesssim (\frac{r}{\de})^{s}$, for any $B_r$ being a ball of radius $r$ $(\de\le r\le 1)$.
\end{enumerate}
We also assume that each $\de$-ball contained in $H$ intersects $\gtrsim |\log\de^{-1}|^{-2}\#\Theta$ many slabs from $\cup_\theta \S_\theta$.
Then 
\[ (\#\Theta)^4\#H\le C_{s,\e} \de^{-2t-s-2-\e}.\] 
\end{theorem}

\subsection{\texorpdfstring{$\d$}{Lg}-discretization of the projection problem}
We show Theorem \ref{linediscrt} implies Theorem \ref{mainline} in this subsection.

\begin{proof}[Proof of Theorem \ref{mainline} assuming Theorem \ref{linediscrt}]
Suppose $A\subset \R^3$ is a Borel set of Hausdorff dimension $\alpha$.
We may assume $A\subset B^3(0,1)$.
Recall the definition of the exceptional set
$$ E_s=\{\theta\in[0,1]:\dim\rho_\theta(A)<s\}. $$
If $\dim(E_s)=0$, then there is nothing to prove. Therefore, we assume $\dim(E_s)>0$.
Recall the definition of the $t$-dimensional Hausdorff content is given by
$$ \cH^t_\infty(B):=\inf\{\sum_i r(B_i)^t:B\subset \cup_i B_i\}. $$
A property for the Hausdorff dimension is that 
$$ \dim(B)=\sup\{t:\cH^t_\infty(B)>0\}. $$
We choose $a<\dim (A),t<\dim (E_s)$. Then $\cH^t_\infty(E_s)>0$, and by Frostman's lemma there exists a probability measure $\nu_A$ supported on $A$ satisfying $\nu_A(B_r)\lesssim r^a$ for any $B_r$ being a ball of radius $r$. We only need to prove 
$$ a\le 2+s-2t, $$
since then we can send $a\rightarrow \dim(A)$ and $t\rightarrow \dim(E_s)$. As $a$ and $t$ are fixed, we may assume $\cH^t_\infty(E_s)\sim 1$ is a constant.

Fix a $\theta\in E_s$. By definition we have $\dim \rho_\theta(A)<s$. We also fix a small number $\en$ which we will later send to $0$.
By Lemma \ref{usefullemma}, we can find a covering of $\rho_\theta(A)$ by intervals $\I_\theta=\{I\}$, each of which has length $2^{-j}$ for some integer $j>|\log_2\en|$. We define $\mathbb I_{\theta,j}:=\{I\in\mathbb I_\theta: r(I)=2^{-j}\}$ (Here $r(I)$ denotes the length of $I$).
Lemma \ref{usefullemma} yields the following properties:
\begin{equation}\label{rsless21}
    \sum_{I\in\mathbb I_\theta}r(I)^s\le 1;
\end{equation}
For each $j$ and $r$-interval $I_r\subset l_\theta$, we have
\begin{equation}\label{structure2}
    \#\{I\in \I_{\theta,j}: I\subset I_r\}\lesssim (\frac{r}{2^{-j}})^s.
\end{equation}

For each $\theta\in E_s$, we can find such a $\I_\theta$. We also define the slab sets $\S_{\theta,j}:=\{\rho^{-1}_\theta(I): I\in\I_{\theta,j}\}\cap B^3(0,1)$, $\S_{\theta}:=\bigcup_j\S_{\theta,j}$. Each slab in $\S_{\theta,j}$ has dimensions $2^{-j}\times 1\times 1$ and normal direction $\gamma(\theta)$. One easily sees that $A\subset \bigcup_{S\in \S_\theta}S$.
By pigeonholing, there exists $j(\theta)$ such that

\begin{equation}\label{pigeon1}
    \nu_A\big(A\cap(\cup_{S\in\S_{\theta,j(\theta)}}S)\big)\ge \frac{1}{10j(\theta)^2}\nu_A(A)=\frac{1}{10j(\theta)^2}.
\end{equation}
For each $j>|\log_2\en|$, define $E_{s,j}:=\{\theta\in E_s: j(\theta)=j\}$. Then we obtain a partition of $E_s$:
$$ E_s=\bigsqcup_j E_{s,j}. $$
By pigeonholing again, there exists $j$ such that
\begin{equation}\label{pigeon2}
    \cH_\infty^t(E_{s,j})\ge \frac{1}{10j^2}\cH_\infty^t(E_s)\sim \frac{1}{10j^2}. 
\end{equation} 
In the rest of the poof, we fix this $j$. We also set $\de=2^{-j}$. By Lemma \ref{lemdelta}, there exists a $(\de,t)$-set $\Theta\subset E_{s,j}$ with cardinality $\#\Theta\gtrsim (\log\de^{-1})^{-2}\de^{-t}$.

Next, we consider the set $U:=\{(x,\theta)\in A\times \Theta: x\in\cup_{S\in\S_{\theta,j}}S \}$. We also use $\mu$ to denote the counting measure on $\Theta$ (note that $\Theta$ is a finite set).
Define the section of $U$:
$$ U_x=\{\theta: (x,\theta)\in U\},\ \ \  U_\theta:=\{x: (x,\theta)\in U\}. $$
By \eqref{pigeon1} and Fubini, we have
\begin{equation}\label{pigeon3}
    (\nu_A\times \mu)(U)\ge \frac{1}{10j^2} \mu(\Theta).
\end{equation}
This implies
\begin{equation}\label{pigeon4}
    (\nu_A\times \mu)\bigg(\Big\{(x,\theta)\in U: \mu(U_x)\ge\frac{1}{20j^2}\mu(\Theta)  \Big\}\bigg)\ge \frac{1}{20j^2} \mu(\Theta),
\end{equation} 
since
\begin{equation}
    (\nu_A\times \mu)\bigg(\Big\{(x,\theta)\in U: \mu(U_x)\le\frac{1}{20j^2}\mu(\Theta)  \Big\}\bigg)\le \frac{1}{20j^2} \mu(\Theta).
\end{equation} 
By \eqref{pigeon4}, we have
\begin{equation}\label{pigeon5}
    \nu_A\bigg(\Big\{x\in A: \mu(U_x)\ge \frac{1}{20j^2}\mu(\Theta) \Big\}\bigg)\ge \frac{1}{20j^2}. 
\end{equation} 

We are ready to apply Theorem \ref{linediscrt}. Recall $\de=2^{-j}$ and $\#\Theta\gtrsim (\log\de^{-1})^{-2}\de^{-t}$. By \eqref{pigeon5} and noting that $\nu_A(B_\de)\lesssim \de^{a}$, we can find a $\de$-separated subset of $\{x\in A: \# U_x\ge \frac{1}{20j^2}\#\Theta \}$ with cardinality $\gtrsim (\log\de^{-1})^{-2}\de^{-a}$. We denote the $\de$-neighborhood of this set by $H$, which is a union of $\de$-balls. For each $\de$-ball $B_\de$ contained in $H$, we see that there are $\gtrsim (\log\de^{-1})^{-2}\#\Theta$ many slabs from $\cup_{\theta\in\Theta}\S_{\theta,j}$ that intersect $B_\de$. We can now apply Theorem \ref{linediscrt} to obtain
$$ (\log\de^{-1})^{-8}\de^{-a-4t}\lesssim(\#\Theta)^4\#H\le C_{s,\e}\de^{-2t-s-2-\e}. $$
Letting $\en\rightarrow 0$ (and hence $\de\rightarrow 0$) and then $\e\rightarrow 0$, we obtain $a\le 2+s-2t$.
\end{proof}

\subsection{Proof of Theorem \ref{linediscrt}}
For convenience, we will prove the following version of Theorem \ref{linediscrt} after rescaling $x\mapsto \de^{-1}x$.
\begin{theorem}\label{linerescale} Fix $0<s<1$. For each $\e>0$, there exists $C_{s,\e}$ so that the following holds. Let $\d>0$. Let $H\subset B^3(0,\de^{-1})$ be a union of $\de^{-a}$ many disjoint unit balls so that $H$ has measure $|H|\sim \de^{-a}$. Let $\Theta$ be a $\de$-separated subset of $[0,1]$ so that $\Theta$ is a $(\de,t)$-set and $\#\Theta\gtrsim (\log\de^{-1})^{-2}\de^{-t}$. Assume for each $\theta\in \Theta$, we have a collection of $1\times \de^{-1}\times \de^{-1}$-slabs $\S_\theta$ with normal direction $\ga(\theta)$. $\S_\theta$ satisfies the $s$-dimensional condition: 
\begin{enumerate}
    \item $\#\S_\theta\lesssim \de^{-s}$, 
    \item $\#\{S\in\S_\theta:S\cap B_r\}\lesssim r^{s}$, for any $B_r$ being a ball of radius $r$ $(1\le r\le \de^{-1})$.
\end{enumerate}
We also assume that each unit ball contained in $H$ intersects $\gtrsim |\log\de^{-1}|^{-2}\#\Theta$ many slabs from $\cup_\theta \T_\theta$.
Then 
\[ (\#\Theta)^4\#H\le C_{s,\e} \de^{-2t-s-2-\e}.\] 
\end{theorem}

We define the cone
\begin{equation}\label{defcone}
    \Ga:=\{r\ga(\theta):1/2\le r\le 1,\theta\in[0,1]\}. 
\end{equation} 
For any large scale $R$, 
there is a standard partition of $N_{R^{-1}}\Ga$ into planks $\si_{R^{-1/2}}$ of dimensions $R^{-1}\times R^{-1/2}\times 1$: $$N_{R^{-1}}\Ga=\bigcup \si_{R^{-1/2}}.$$
Here, the subscript of $\si_{R^{-1/2}}$ denotes its angular size.
For any function $f$ and plank $\si=\s_{R^{-1/2}}$, we define $f_\si:=(1_\si\wh f)^\vee$ as usual. The main tool we need is the following \textit{fractal small cap decoupling} for the cone $\Ga$.

\begin{theorem}[fractal small cap decoupling]\label{fractaldec}
Suppose $N_\de(\Ga)=\bigcup \ga$, where each $\ga$ is a $\de\times\de\times 1$-cap. 
Given a function $g$, we say $g$ is $t${\bf -spacing} if $\supp\wh g\subset \cup_{\ga\in\Ga_g}\ga$, where $\Ga_g$ is a set of $\de\times\de\times 1$-caps from the partition of $N_\de(\Ga)$ and satisfies:
\begin{equation}\label{tspacing}
    \#\{\ga\in\Ga_g: \ga\subset \si_r\}\lesssim (r/\de)^t, \textup{~for~any~}r^2\times r\times 1-\textup{plank~} \si_r\subset N_{r^2}\Ga\ \  (\de\le r\le 1). 
\end{equation} 
If $g$ is $t$-spacing, then we have
\begin{equation}\label{fractaldecineq}
    \int_{B_{\de^{-1}}} |g|^{4}\lesssim_\e \de^{-\e-t}\sum_\ga \int |g_\ga|^{4}.
\end{equation}
\end{theorem}

Small cap decoupling for the cone was studied by the second and third authors in \cite{ampdep}, where they proved amplitude-dependent versions of the wave envelope estimates (Theorem 1.3) of \cite{guth2020sharp}. Wave envelope estimates are a more refined version of square function estimates, and sharp small cap decoupling is a straightforward corollary. For certain choices of conical small caps, the critical $L^{p_c}$ exponent is $p_c=4$ (as is the case in our Theorem \ref{fractaldec}). When $p_c=4$, the sharp small cap decoupling inequalities follow already from the wave envelope estimates of \cite{guth2020sharp}. A version of this was first observed in Theorem 3.6 of \cite{demeter2020small} and was later thoroughly explained in \textsection10 of \cite{ampdep}. To prove Theorem \ref{fractaldec} above, we repeat the derivation of small cap decoupling from the wave envelope estimates of \cite{guth2020sharp} but incorporate the extra ingredient of $t$-spacing.

\begin{remark}\label{closeconerm}
{\rm
We will actually apply Theorem \ref{fractaldec} to a slightly different cone \begin{equation}\label{defclosecone}
    \Ga_{K^{-1}}=\{r\ga(\theta):K^{-1}\le r\le 1, \theta\in[0,1]\}. 
\end{equation} 
Compared with $\Ga$, we see that $\Ga_{K^{-1}}$ is at distance $K^{-1}$ from the origin, but we still have a similar fractal small cap decoupling for $\Ga_{K^{-1}}$. Instead of \eqref{fractaldecineq}, we have
\begin{equation}\label{closeconedec}
    \int_{B_{\de^{-1}}} |g|^{4}\lesssim_\e K^{O(1)}\de^{-\e-t}\sum_\ga \int |g_\ga|^{4}.
\end{equation}
The idea is to partition $\Ga_{K^{-1}}$ into $\sim K$ many parts, each of which is roughly a cone that we can apply Theorem \ref{fractaldec} to. By triangle inequality, it gives \eqref{closeconedec} with an additional factor $K^{O(1)}$.
It turns out that this factor is not harmful, since we will set $K\sim (\log\de^{-1})^{O(1)}$ which can be absorbed into $\de^{-\e}$.
}
\end{remark}

We postpone the proof of Theorem \ref{fractaldec} to the next subsection, and first show how it implies Theorem \ref{linerescale}.

\begin{proof}[Proof of Theorem \ref{linerescale} assuming Theorem \ref{fractaldec}]
We consider the dual of each $S_\theta\in\S_\theta$ in the frequency space.
For each $\theta\in\Theta$, we define $\tau_\theta$ to be a tube centered at the origin that has dimensions $\de\times\de\times 1$, and its direction is $\ga(\theta)$. We see that $\tau_\theta$ is the dual of each $S_\theta\in\S_\theta$. Now, for each $S_\theta\in\S_\theta$, we choose a bump function $\psi_{S_\theta}$ satisfying the following properties: $\psi_{S_\theta}\ge 1$ on $S_\theta$, $\psi_{S_\theta}$ decays rapidly outside $S_\theta$, and $\supp \wh\psi_{S_\theta}\subset \tau_\theta$.

Define functions 
\[ f_\theta=\sum_{S_\theta\in\S_\theta}\s_{S_\theta}\qquad\text{and}\qquad f=\sum_{\theta\in \Theta}f_\theta. \]
From our definitions, we see that for any $x\in H$, we have $f(x)\gtrsim (\log\de^{-1})^{-2}\#\Theta$. Therefore, we obtain
\begin{equation}\label{line1}
    |H|(\#\Theta)^4\lessapprox\int_{H}|f|^4,
\end{equation}
where ``$\lessapprox$" means ``$\lesssim (\log\de^{-1})^{O(1)}$". 

Next, we will do a high-low decomposition for each $\tau_\theta$.
\begin{definition}
Let $K$ be a large number which we will choose later. Define the high part of $\tau_\theta$ as
$$ \tau_{\theta,high}:=\{\xi\in\tau_\theta: K^{-1}\le|\xi\cdot\ga(\theta)|\le 1\}. $$
Define the low part of $\tau_\theta$ as 
$$ \tau_{\theta,low}:=\tau_\theta\setminus \tau_{\theta,high}= \{\xi\in\tau_\theta:|\xi\cdot\ga(\theta)|\le K^{-1}\}.$$
\end{definition}

We choose a smooth partition of unity adapted to the covering $\tau_\theta=\tau_{\theta,high}\bigcup\tau_{\theta,low}$ which we denote by $\eta_{\theta,high}, \eta_{\theta,low}$, so that
$$ \eta_{\theta,high}+\eta_{\theta,low}=1 $$
on $\tau_{\theta}$. The key observation is that $\{\supp \wh \eta_{\theta,high}\}_\theta$ are at most $O(K)$-overlapping and form a canonical covering of $N_\de (\Ga_{K^{-1}})$. (See the definition of $\Ga_{K^{-1}}$ in \eqref{defclosecone}).

Since $\supp\wh f_\theta\subset \tau_\theta$, we also obtain a decomposition of $f_\theta$
\begin{equation}
    f_\theta=f_{\theta,high}+f_{\theta,low},
\end{equation}
where $\wh f_{\theta,high}=\eta_{\theta,high} \wh f_\theta, \wh f_{\theta,low}=\eta_{\theta,low}\wh f_{\theta}.$
Similarly, we have a decomposition of $f$
\begin{equation}
    f=f_{high}+f_{low},
\end{equation}
where $f_{high}=\sum_{\theta}f_{\theta,high},  f_{low}=\sum_\theta f_{\theta,low}.$

Recall that for $x\in H$, we have
$$(\log\de^{-1})^{-2}\#\Theta\lesssim f(x)\le |f_{high}(x)|+|f_{low}(x)|. $$
We will show that by properly choosing $K$, we have  
\begin{equation}\label{pointwise}
    |f_{low}(x)|\le C^{-1}(\log\de^{-1})^{-2}\#\Theta.
\end{equation}
Recall that $f_{low}=\sum_\theta f_\theta*\eta^\vee_{\theta,low}$. Since $\eta_{\theta,low}$ is a bump function at $\tau_{\theta,low}$, we see that $\eta_{\theta,low}^\vee$ is an $L^1$-normalized bump function essentially supported in the dual of $\tau_{\theta,low}$. Denote the dual of $\tau_{\theta,low}$ by $S_{\theta,K}$ which is a $\de^{-1}\times\de^{-1}\times K$-slab whose normal direction is $\ga(\theta)$. One actually has
$$ |\eta^\vee_{\theta,low}|\lesssim \frac{1}{|S_{\theta,K}|}\psi_{S_{\theta,K}}. $$
Here, $\psi_{S_{\theta,K}}$ is bump function $=1$ on $S_{\theta,K}$ and decays rapidly outside $S_{\theta,K}$.
Ignoring the rapidly decaying tails, we have
$$ |f_{low}(x)|\lesssim \sum_\theta \frac{1}{K}\#\{ S_\theta\in\S_\theta:S_\theta\cap B_{100K}(x)\neq \emptyset \}. $$
Recalling the condition (2) in Theorem \ref{linerescale}, we have
$$ \#\{S_\theta\in\S_\theta: S_\theta\cap B_{100K}(x)\neq\emptyset\}\lesssim (100K)^s. $$
This implies
$$ |f_{low}(x)|\lesssim K^{s-1}\#\Theta. $$
Since $s<1$, by choosing $K\sim (\log\de^{-1})^{\frac{2}{1-s}}$, we obtain \eqref{pointwise}. This shows that for $s\in H$, we have
$$ (\log\de^{-1})^{-2}\#\Theta\lesssim |f(x)|\lesssim |f_{high}|. $$

We define $g=f_{high}$. By remark \eqref{closeconerm}, we actually see that $\{\tau_{\theta,high}\}$ form a $K$-overlapping covering of $N_\de(\Ga_K)$, 
and we have the decoupling inequality \eqref{closeconedec}.
By \eqref{line1}, we have
$$ |H|\de^{-4t}\lesssim \int_H |f|^4\lesssim \int_{B_{\de^{-1}}}|f_{high}|^4. $$
By \eqref{closeconedec}, it is further bounded by
$$ \lesssim_\e \de^{-t-\e}\sum_{\theta} \int |f_{\theta,high}|^4\lesssim \de^{-t-\e}\sum_\theta \int |\sum_{S_\theta\in\S_\theta} \psi_{S_\theta}|^4. $$
Since the slabs in $\S_\theta$ are essentially disjoint, the above expression is bounded by
$$\lesssim \de^{-t-\e}\sum_\theta\int\sum_{S_\theta\in\S_\theta}|\psi_{S_\theta}|^4\sim \de^{-t-\e}\sum_\theta\sum_{S_\theta\in\S_\theta}|S_\theta|\sim \de^{-t-\e} \de^{-s-t-1}.$$
This implies $(\#\Theta)^4\# H\lesssim_\e \de^{2t-s-1-\e}$.

\end{proof}

\subsection{Proof of Theorem \ref{fractaldec}}
The proof of Theorem \ref{fractaldec} is based on an inequality of Guth, Wang and Zhang. Let us first introduce some notation from their paper \cite{guth2020sharp}.
Let $\Gn$ denote the standard cone in $\R^3$:
$$\Gn:=\{(r\cos\theta,r\sin\theta,r):1/2\le r\le 1,\theta\in[0,2\pi]\}.$$
We can partition the $\de$-neighborhood of $\Gn$ into $\de\times \de^{1/2}\times 1$-planks $\Si=\{\si\}$: $$N_{\de}(\Gn)=\bigsqcup \si.$$
More generally, for any dyadic $s$ in the range $\de^{1/2}\le s\le 1$, we can partition the $s^2$-neighborhood of $\Gn$ into $s^2\times s\times 1$-planks $\Tau_s=\{\tau_s\}$: $$N_{s^2}(\Gn)=\bigsqcup \tau_s.$$
For each $s$ and frequency plank $\tau_s\in \Tau_s$, we define the box $U_{\tau_s}$ in the physical space to be a rectangle centered at the origin of dimensions $\de^{-1}\times \de^{-1}s\times \de^{-1}s^2$ whose edge of length $\de^{-1}$ (respectively $\de^{-1}s$, $\de^{-1}s^2$) is parallel to the edge of $\tau_s$ with length $s^2$ (respectively $s$, $1$). Note that for any $\si\in\Si$, $U_\si$ is just the dual rectangle of $\si$. Also, $U_{\tau_s}$ is the convex hull of $\cup_{\si\subset \tau_s}U_{\si}$.

If $U$ is a translated copy of $U_{\tau_s}$, then we define $S_U f$ by
\begin{equation}
    S_U f=\big(\sum_{\si\subset \tau_s}|f_\si|^2\big)^{1/2}\Id_U.
\end{equation}
We can think of $S_U f$ as the wave envelope of $f$ localized in $U$ in the physical space and localized in $\tau_s$ in the frequency space.
We have the following inequality of Guth, Wang and Zhang (see \cite{guth2020sharp} Theorem 1.5):
\begin{theorem}[Wave envelope estimate]\label{GWZstandard}
Suppose $\supp\wh f\subset N_{\de}(\Gn)$. Then
\begin{equation}
    \|f\|_{4}^4\le C_\e \de^{-\e} \sum_{\de^{1/2}\le s\le 1}\sum_{\tau_s\in\Tau_s}\sum_{U\parallel U_{\tau_s}} |U|^{-1}\|S_Uf\|_2^4.
\end{equation}
\end{theorem}

Although the theorem above is stated for the standard cone $\Ga_\circ$, it is also true for general cone $\Ga$ (see \eqref{defcone}). The appendix of \cite{ampdep} shows how to adapt the inductive proof of Guth-Wang-Zhang for $\Ga_\circ$ to the case of a general cone $\Ga$.

As we did for $\Ga_\circ$, we can also define the $\de\times \de^{1/2}\times 1$-planks $\Si=\{\si\}$ and $s^2\times s\times 1$-planks $\Tau_s=\{\tau_s\}$, which form a partition of certain neighborhood of $\Ga$. We can similarly define the wave envelope $S_U f$ for $\supp\wh f\subset N_{\de}(\Ga)$. We have the following estimate for general cone.

\begin{theorem}[Wave envelope estimate for general cone]\label{GWZ}
Suppose $\supp\wh f\subset N_{\de}(\Ga)$. Then
\begin{equation}
    \|f\|_{4}^4\le C_\e \de^{-\e} \sum_{\de^{1/2}\le s\le 1}\sum_{\tau_s\in\Tau_s}\sum_{U\parallel U_{\tau_s}} |U|^{-1}\|S_Uf\|_2^4.
\end{equation}
\end{theorem}
We are ready to prove Theorem \ref{fractaldec}.

\begin{proof}[Proof of Theorem \ref{fractaldec}]
By pigeonholing, we can assume all the wave packet of $g_\ga$ have amplitude $\sim 1$, so we have
\begin{equation}\label{L4L2}
    \int |g_\ga|^4\sim\int |g_\ga|^2.
\end{equation}
Apply Theorem \ref{GWZ} to $g$, we have
$$ \|g\|_{4}^4\le C_\e \de^{-\e} \sum_{\de^{1/2}\le s\le 1}\sum_{\tau_s\in\Tau_s}\sum_{U\parallel U_{\tau_s}} |U|^{-1}\|S_Ug\|_2^4. $$
For fixed $s, \tau_s, U\parallel U_{\tau_s}$, let us analyze the quantity $\|S_Ug\|_2^2$ on the right hand side.
By definition,
$$\|S_Ug\|_2^2=\int_U \sum_{\si\subset\tau_s}|g_\si|^2. $$
Note that $U$ has dimensions $\de^{-1}\times \de^{-1}s\times \de^{-1}s^2$, so its dual $U^*$ has dimensions $\de\times \de s^{-1}\times \de s^{-2}$. We will apply local orthogonality to each $f_\si$ on $U$. Let $\{\beta\}$ be a set of $(\de s^{-1})^2\times \de s^{-1}\times 1$-planks that form a finitely overlapping covering of $N_{\de s^{-1}}(\Ga)$. We see that $U^*$ and each $\beta$ have the same angular size $\de s^{-1}$.
For reader's convenience, we recall that we have defined three families of planks: $\{\ga:\ga\in\Ga_g\}$ of dimensions $\de\times \de\times 1$; $\{\beta\}$ of dimensions $(\de s^{-1})^2\times \de s^{-1}\times 1$; $\{\si\}$ of dimensions $\de\times \de^{1/2}\times 1$.

Since $\de^{1/2}\le s\le 1$, we have the nested property for these planks: each $\ga(\in \Ga_g)$ is contained in $100$-dilation of some $\beta$ and each $\beta$ is contained in $100$-dilation of some $\si$. We simply denote this relationship by $\ga\subset \beta, \beta\subset \si$. We can write 
$$ g_\si=\sum_{\beta\subset \si} g_\beta,\ \ \ g_\beta=\sum_{\ga\subset \beta} g_\ga. $$

Choose a smooth bump function $\psi_U$ at $U$ satisfying: $|\psi_U|\gtrsim \Id_U$, $\psi_U$ decays rapidly outside $U$, and $\wh \psi_U$ is supported in $U^*$. 
We have
$$ \int_U |g_\si|^2\lesssim \int |\psi_U \sum_{\beta\subset \si}g_\beta|^2.  $$
Since $(\psi_U g_\beta)^\wedge\subset U^*+\beta$ and by a geometric observation that $\{U^*+\beta\}_{\beta\subset \si}$ are finitely overlapping, we have
$$ \int_U|g_\si|^2\lesssim \int \sum_{\beta\subset\si}|\psi_U g_\beta|^2=\int \sum_{\beta\subset\si}|\psi_U \sum_{\ga\subset \beta}g_\ga|^2\lesssim \int \sum_{\beta\subset\si}\#\{\ga\subset \beta\}\sum_{\ga\subset \beta}|\psi_U g_\ga|^2. $$
Summing over $\si\subset \tau_s$, we get
\begin{align}
    \|S_U g \|_2^2=\int_U \sum_{\si\subset\tau_s}|g_\si|^2&\lesssim \int \sum_{\si\subset\tau_s}\sum_{\beta\subset\si}\#\{\ga\subset \beta\}\sum_{\ga\subset \beta}|\psi_U g_\ga|^2\\
    &\lesssim \big(\sum_{\si\subset\tau_s}\sum_{\beta\subset\si}\#\{\ga\subset \beta\}\big)(\sup_\beta\sum_{\ga\subset \beta}|g_\ga|^2)\int |\psi_U|^2\\
    (\textup{Since~}\|g_\ga\|_\infty\le 1)\ \ \ &\lesssim \#\{\ga\subset\tau_s\} \#\{\ga\subset\beta\}|U|\\
    (\textup{By~the~}t\textup{-spacing~condition})\ \ \ &\lesssim (s/\de)^t(\de s^{-1}/\de)^t |U|\\
    &=\de^{-t}|U|.
\end{align} 

Therefore, we have
\begin{equation}
    \sum_{U}|U|^{-1}\|S_U g\|_2^4\lesssim \de^{-t}\sum_U \|S_U g\|_2^2=\de^{-t}\int \sum_{\si\subset\tau_s}|g_\si|^2\lesssim\de^{-t}\int \sum_{\ga\subset\tau_s}|g_\ga|^2.
\end{equation}
The last inequality is by the $L^2$ orthogonality.

Noting \eqref{L4L2}, we have
\begin{equation}
    \|g\|_{4}^4\le C_\e \de^{-\e-t} \sum_{\de^{1/2}\le s\le 1}\sum_\ga \int|g_\ga|^2\sim C_\e \de^{-\e-t} \sum_{\de^{1/2}\le s\le 1}\sum_\ga \int|g_\ga|^4\lesssim C_\e \de^{-2\e-t}\sum_\ga \int|g_\ga|^4.
\end{equation}

\end{proof}

\bibliographystyle{abbrv}
\bibliography{bibli}

\end{document}